\documentclass[a4paper]{emsprocart}

\contact[behrndt@tugraz.at]{Jussi Behrndt,
Institut f\"{u}r Numerische Mathematik,
Technische Universit\"{a}t Graz,
Steyrergasse 30,
8010 Graz, Austria}

\contact[m.langer@strath.ac.uk]{Matthias Langer,
Department of Mathematics and Statistics,
University of Strathclyde,
26 Richmond Street,
Glasgow G1 1XH, United Kingdom}

\contact[lotoreichik@ujf.cas.cz]{Vladimir Lotoreichik,
Department of Theoretical Physics,
Nuclear Physics Institute, Academy of Sciences, 250 68,
\v{R}e\v{z} near Prague, Czech Republic}

\usepackage{lmodern}
\usepackage{eucal}
\usepackage{dsfont}
\usepackage[latin1]{inputenc}
\usepackage[T1]{fontenc}

\numberwithin{figure}{section}
\numberwithin{equation}{section}
\theoremstyle{plain}
\newtheorem{theorem}{Theorem}[section]
\newtheorem{hypothesis}{Hypothesis}[section]
\newtheorem{lemma}[theorem]{Lemma}
\newtheorem{proposition}[theorem]{Proposition}
\newtheorem{definition}[theorem]{Definition}
\theoremstyle{remark}
\newtheorem{remark}[theorem]{Remark}

\newcommand\frk{\mathfrak{k}}
\newcommand\frh{\mathfrak{h}}
\newcommand\rmO{\mathrm{O}}
\newcommand\uhr{\upharpoonright}
\newcommand\dd{{\mathsf{d}}}
\newcommand\sfH{\mathsf{H}}
\newcommand\sfK{\mathsf{K}}
\newcommand\lm{{\lambda}}
\newcommand\arr{\rightarrow}
\newcommand\cB{\mathcal B}
\newcommand\cG{\mathcal G}
\newcommand\cH{\mathcal H}
\newcommand\cK{\mathcal K}
\newcommand\CC{\mathbb C}
\newcommand\NN{\mathbb N}
\newcommand\dC{{\mathbb C}}
\newcommand\dN{{\mathbb N}}
\newcommand\dR{{\mathbb R}}
\newcommand\cU{{\mathcal U}}
\newcommand\sA{{\mathfrak A}}

\newcommand\sS{{\mathfrak S}}
\newcommand\frA{\mathfrak A}
\newcommand\frS{\mathfrak S}

\newcommand\ov{\overline}
\newcommand\wt{\widetilde}
\newcommand\wh{\widehat}
\newcommand{\defeq}{\mathrel{\mathop:}=}

\DeclareMathOperator\Tr{Tr}
\DeclareMathOperator\dom{dom}
\DeclareMathOperator\ran{ran}

\newcommand\Op{\sfH_{\alpha,\Sigma}}
\newcommand\Opp{\sfK_{\omega,\Sigma}}
\newcommand\Opfree{\sfH_{\rm free}}
\newcommand\OpN{\sfK_{\rm N}}
\newcommand\AB{A_{[B]}}

\newcounter{counter_a}
\newenvironment{myenum}{\begin{list}{{\rm(\roman{counter_a})}}%
{\usecounter{counter_a}
\setlength{\itemsep}{1.ex}\setlength{\topsep}{0.8ex}
\setlength{\leftmargin}{5ex}\setlength{\labelwidth}{5ex}}}{\end{list}}

\title[Trace formulae for Schr\"odinger operators
with singular interactions]{Trace formulae for Schr\"odinger operators with \\
singular interactions}

\author[Jussi Behrndt, Matthias Langer and Vladimir Lotoreichik]{Jussi Behrndt,
Matthias Langer and Vladimir Lotoreichik\thanks{JB gratefully acknowledges financial support
by the Austrian Science Fund (FWF): Project P~25162-N26.
VL gratefully acknowledges financial support by the Czech Science Foundation
(GA\v{C}R): Project 14-06818S.}}

\begin{document}

\begin{center}
\textit{Dedicated with great pleasure to our teacher, colleague and friend \\
Pavel Exner on the occasion of his 70th birthday.}
\end{center}

\begin{abstract}
Let $\Sigma\subset\dR^d$ be a $C^\infty$-smooth
closed compact hypersurface, which
splits the Euclidean space $\dR^d$ into two domains $\Omega_\pm$.
In this note self-adjoint Schr\"odinger operators
with $\delta$ and $\delta'$-interactions supported on $\Sigma$ are studied.
For large enough $m\in\dN$ the difference of
$m$th powers of resolvents of such a Schr\"odinger operator and the free Laplacian
is known to belong to the trace class.
We prove trace formulae, in which the trace of the resolvent power
difference in $L^2(\dR^d)$ is written
in terms of Neumann-to-Dirichlet maps on the boundary space $L^2(\Sigma)$.
\end{abstract}



\maketitle

\section{Introduction}

This paper is strongly inspired by the work of Pavel Exner on
Schr\"odinger operators with singular interactions
of $\delta$ and $\delta'$-type supported on hypersurfaces in $\dR^d$.
Such operators play an important role in mathematical physics,
for instance, in nuclear physics or solid state physics
or in connection with photonic crystals or other nanostructures.
In the case of a curve in $\dR^2$ such models are also called ``leaky quantum wires''.
The first rigourous investigations of such operators started in the late 1980s
(see, e.g.\ \cite{AGS87, BEKS94, BT92}), and the interest in these operators
grew steadily in the last two decades.
We refer the reader to the review paper \cite{E08}, the monograph \cite{EK15},
the references therein and also to the more recent papers
\cite{BEL14_RMP, BGLL15, BLL13_AHP, DEKP15, DR14, EJ13, JL15, LO15}.

Let $\Sigma\subset\dR^d$, where $d\ge2$, be a $C^\infty$-smooth closed compact
hypersurface without boundary, which naturally splits the Euclidean space $\dR^d$
into a bounded domain $\Omega_-$ and an exterior domain $\Omega_+$.
Moreover, let $\alpha,\omega\in L^\infty(\Sigma)$ be real-valued functions.
The Schr\"odinger operator $\Op$ with $\delta$-interaction of strength $\alpha$
and the Schr\"odinger operator $\Opp$ with $\delta'$-interaction of strength $\omega$
are formally given by
\begin{equation}\label{formal_operators}
	-\Delta - \alpha\delta(x-\Sigma)
	\qquad\text{and}\qquad
	-\Delta - \omega\delta'(x-\Sigma).
\end{equation}
We define these operators rigorously via quadratic forms;
see Definition~\ref{def:Ops} below.
Let us first fix some notation.
Since the space $L^2(\dR^d)$ naturally decomposes as
$L^2(\dR^d) = L^2(\Omega_+)\oplus L^2(\Omega_-)$,
we can write functions $u\in L^2(\dR^d)$ as $u = u_+ \oplus u_-$
with $u_\pm = u\uhr\Omega_\pm\in L^2(\Omega_\pm)$.
The $L^2$-based Sobolev spaces of order $s \ge 0$ over $\dR^d$ and $\Omega_\pm$
are denoted by $H^s(\dR^d)$ and $H^s(\Omega_\pm)$, respectively.
Note that the hypersurface $\Sigma$ coincides with the boundaries $\partial\Omega_\pm$
of the domains $\Omega_\pm$.
Hence, for any $u\in H^1(\dR^d)$ and $u_\pm \in H^1(\Omega_\pm)$ the \emph{traces}
$u|_\Sigma$ and $u_\pm|_\Sigma$ on $\Sigma$ are well defined as functions in $L^2(\Sigma)$.
Further, for a function $u\in H^1(\dR^d\setminus \Sigma) \defeq H^1(\Omega_+)\oplus H^1(\Omega_-)$
we define its \emph{jump} on $\Sigma$ as $[u]_\Sigma \defeq u_+|_{\Sigma} - u_-|_{\Sigma}$.

Let us now introduce the following quadratic forms that correspond
to the formal expression in \eqref{formal_operators}.
According to~\cite[\S2]{BEKS94}, \cite[\S3.4]{BLL13_AHP}
and \cite[Proposition~3.1]{BEL14_RMP}, the symmetric quadratic forms
\begin{alignat*}{2}
	\frh_{\alpha,\Sigma}[u] &\defeq
	\|\nabla u\|^2 - (\alpha u|_{\Sigma}, u|_{\Sigma})_\Sigma, \qquad &
	\dom\frh_{\alpha,\Sigma} &\defeq H^1(\dR^d),
	\\[0.5ex]
	\frk_{\omega,\Sigma}[u]
	&\defeq
	\|\nabla u_+\|_+^2 + \|\nabla u_-\|_-^2 -(\omega [u]_{\Sigma}, [u]_{\Sigma})_\Sigma,
	\qquad &
	\dom\frk_{\omega, \Sigma} &\defeq H^1(\dR^d \setminus \Sigma),
\end{alignat*}
in $L^2(\dR^d)$ are closed, densely defined and bounded from below;
here $u = u_+ \oplus u_-$ with $u_\pm$ as above, and $\|\cdot\|_\pm$ denotes
the norm on $L^2(\Omega_\pm;\dC^d)$.

\begin{definition}\label{def:Ops}
	Let $\Op$ and $\Opp$ be the self-adjoint operators in $L^2(\dR^d)$
	corresponding to the forms $\frh_{\alpha,\Sigma}$ and $\frk_{\omega,\Sigma}$, respectively,
	via  the first representation theorem {\rm(\cite[Theorem~VI.2.1]{K})}.
	Moreover, set $\Opfree \defeq \sfH_{0,\Sigma}$ {\rm(}$\alpha \equiv 0${\rm)}
	and $\OpN \defeq \sfK_{0,\Sigma}$ {\rm(}$\omega \equiv 0${\rm)}.
\end{definition}
The operator $\Op$ is called Schr\"odinger operator with $\delta$-interaction
of strength $\alpha$ supported on $\Sigma$;
the operator $\Opp$ is called Schr\"odinger operator
with $\delta'$-interaction of strength%
\footnote{We point out that, in the case of invertible $\omega$, not $\omega$
itself, but its inverse is frequently called the strength of the $\delta'$-interaction.}
$\omega$ supported on $\Sigma$.
The operator $\Opfree$ is the usual \emph{free Laplacian} on $\dR^d$,
and $\OpN$ is the orthogonal sum of the standard \emph{Neumann Laplacians}
on $\Omega_+$ and $\Omega_-$.
Let us mention that the operators $\Op$ and $\Opp$ can also be introduced
via interface conditions at the hypersurface $\Sigma$; see, e.g.\ \cite{BLL13_AHP}.

The aim of this paper is to derive trace formulae for $\Op$ and $\Opp$.
According to \cite{BLL13_AHP} for $m\in\dN$ the resolvent power differences
\begin{equation}\label{resolv_pwr_diff}
\begin{alignedat}{2}
	& (\Op - \lm)^{-m} - (\Opfree - \lm)^{-m}, \qquad && m > \frac{d - 2}{2}\,,\quad
	\lm\in\rho(\Op), \\[1ex]
	& (\Opp - \lm)^{-m} - (\Opfree - \lm)^{-m},\qquad && m > \frac{d - 1}{2}\,,\quad
	\lm\in\rho(\Opp),
\end{alignedat}
\end{equation}
are in the trace class.
Their traces as functions of $\lm$ or as functions of interaction
strengths are expected to encode
a lot of information on the operators $\Op$ and $\Opp$ themselves and on
the shape of $\Sigma$.
Such non-trivial connections have been observed in various other
settings in the classical papers \cite{BF60, JP51, LP69}
and more recently in, e.g.\ \cite{BGR82, G84, KS03, KS09, M88}.

The main results of the paper (see Theorems~\ref{thm1} and \ref{thm2})
are formulae that express the traces of the resolvent
power differences in \eqref{resolv_pwr_diff}
in terms of traces of derivatives of certain operator-valued functions in the
boundary space $L^2(\Sigma)$.  These operator-valued functions are, in turn,
expressed in terms of Neumann-to-Dirichlet maps on $\Omega_\pm$
corresponding to the differential expression $-\Delta-\lm$ and
in terms of the coupling functions $\alpha$, $\omega$.
Trace formulae of this kind are useful (see, e.g.\ \cite{Ca02, CGNZ12, GZ12})
in connection with the estimation of the spectral shift function.

\subsection{Traces, Neumann-to-Dirichlet maps and some operator functions}

We first recall some notions that are needed in order to formulate the main
results of this paper.
For a compact operator $K$ in a Hilbert space $\cH$ we define its
\emph{singular values} $s_k(K)$, $k=1,2,\dots$, as the eigenvalues of
the non-negative compact operator $|K| = (K^*K)^{1/2} \ge 0$ in $\cH$
ordered in non-decreasing way and with multiplicities taken into account.
If $\sum_{k =1}^\infty s_k(K) < \infty$, we say that $K$ belongs to the \emph{trace class}
and define its \emph{trace} as
\[
	\Tr K \defeq \sum_{k=1}^\infty \lm_k(K),
\]
where $\lm_k(K)$ are the eigenvalues of $K$ repeated with their \emph{algebraic multiplicities}.
Note also that the series in the definition of the trace converges absolutely.

Let us also define some auxiliary maps associated with
partial differential equations. For the sake of brevity, we introduce the spaces
\begin{equation}\label{eq:H32pm}
	H^{3/2}_\Delta(\Omega_\pm) \defeq \big\{u_\pm\in H^{3/2}(\Omega_\pm):
	\Delta u_\pm \in L^2(\Omega_\pm)\big\}.
\end{equation}
For any $u_\pm \in H^{3/2}_\Delta(\Omega_\pm)$ its \emph{Neumann trace}
$\partial_{\nu_\pm} u_\pm|_{\Sigma}$ exists as a function in $L^2(\Sigma)$;
see, e.g.\ \cite[\S2.7.3]{LM72-I}.
For every $\lm\in\dC\setminus\dR_+$ (where $\dR_+\defeq[0,\infty)$)
and every $\varphi \in L^2(\Sigma)$ the boundary value problems
\begin{alignat*}{2}	
		-\Delta u_\pm  &= \lm u_\pm  \qquad &&\text{in}~\Omega_\pm, \\[0.5ex]
		\partial_{\nu_\pm} u_\pm\big|_{\Sigma} &= \varphi  \qquad &&\text{on}~\Sigma,
\end{alignat*}
have unique solutions $u_{\lm,\pm}(\varphi) \in H^{3/2}_\Delta (\Omega_\pm)$;
see, e.g.\ \cite[\S2.7.3]{LM72-I}.
The operator-valued functions $\lm \mapsto M_\pm(\lm)$, $\lm\in\dC\setminus\dR_+$,
are then defined as
\[
	M_\pm (\lm): L^2(\Sigma) \arr L^2(\Sigma),\qquad
	M_\pm(\lm)\varphi \defeq u_{\lm,\pm}(\varphi)|_\Sigma.
\]
For fixed $\lm\in\dC\setminus\dR_+$ the operators $M_\pm(\lm)$ are
the \emph{Neumann-to-Dirichlet maps} for the differential
expression $-\Delta - \lm$ on the domains $\Omega_\pm$.
The operators $M_\pm(\lm)$ are compact and injective and their
inverses are called \emph{Dirichlet-to-Neumann maps}.
Recently, there has been a considerable growth of interest
in the investigation of these maps (see, e.g.\ \cite{AtE11, AtE15, tEO14}),
in particular also with the aim to derive spectral properties of the
corresponding partial differential operators (see, e.g.\ \cite{AM12, BR15, Fr03}).

Further, we define the following operator-valued functions
$\lm \mapsto \wt M(\lm), \wh M(\lm)$, $\lm\in\dC\setminus\dR_+$, by
\begin{equation}\label{def:M}
	\wt M(\lm) \defeq \big(M_+(\lm)^{-1} + M_-(\lm)^{-1}\big)^{-1},
	\qquad
	\wh M(\lm) \defeq M_+(\lm) + M_-(\lm).
\end{equation}
We should mention that for every $\lm\in\dC\setminus\dR_+$ the
operator $M_+(\lm)^{-1} + M_-(\lm)^{-1}$ is invertible
and therefore $\wt M(\lm)$ is well defined.
Moreover, $\wt M(\lm)$ and $\wh M(\lm)$ are compact operators
in $L^2(\Sigma)$ for every $\lm\in\dC\setminus\dR_+$;
see \cite[Propositions~3.2 and 3.8]{BLL13_AHP}.
It is worth mentioning that $\wt M(\lm)$  and the inverse of $\wh M(\lm)$
appear naturally in the theory of boundary integral operators.
They are used in the treatment of partial differential equations from
both analytical~\cite{McL} and computational~\cite{St} viewpoints.
The operator-valued function $\wt M(\cdot)$ was successfully applied to the
spectral analysis of the operator $\Op$ in quite a few papers;
see, e.g.\ \cite{EHL06, EI01, EK15, KL14} and the survey paper \cite{E08}.

\subsection{Statement of the main results}

In the first main result of this note
we obtain a trace formula for the resolvent power difference of the
operators $\Op$ and $\Opfree$.

\begin{theorem}\label{thm1}
	Let the self-adjoint operators $\Opfree$ and $\Op$
	with $\alpha \in L^\infty(\Sigma;\dR)$	be as in Definition~\ref{def:Ops},
	and let the operator-valued function $\wt M$ be as in~\eqref{def:M}.
	Then for all $m \in \dN$ such that $m  > \frac{d-2}{2}$
	 and all $\lambda\in\rho(\Op)$
	the resolvent power difference
	\[
		\wt D_{\alpha,m}(\lm) \defeq (\Op-\lm)^{-m} - (\Opfree-\lm)^{-m}
	\]
	belongs to the trace class, and its trace can be expressed as
	\[
		\Tr\big(\wt D_{\alpha, m}(\lm)\big)
		=
		\frac{1}{(m-1)!}
		\Tr\Bigg(\frac{\dd^{m-1}}{\dd\lm^{m-1}}
		\Big(\big(I -\alpha \wt M(\lm)\big)^{-1}\alpha\wt M'(\lambda)\Big)\Bigg).
	\]
\end{theorem}

\medskip

\noindent
In the second main result of this note we obtain trace formulae for the
resolvent power differences
of the pairs of operators $\{\Opp, \OpN\}$ and $\{\Opp, \Opfree\}$.

\begin{theorem}\label{thm2}
	Let the self-adjoint operators $\Opfree$, $\OpN$ and $\Opp$
	with $\omega\in L^\infty(\Sigma;\dR)$ be as in Definition~\ref{def:Ops},
	and let the operator-valued function $\wh M$ be as in~\eqref{def:M}.
	Then the following statements hold.
	\begin{myenum}
		\item For all $m \in \dN$ such that
		$m  > \frac{d-2}{2}$ and all $\lm\in\rho(\Opp)$
		the resolvent power difference
		\[
			\wh E_{\omega, m}(\lm) \defeq
			(\Opp -\lm)^{-m} - (\OpN-\lambda)^{-m}
		\]
		belongs to the trace class, and its trace can be expressed as
		\[
			\Tr\big(\wh E_{\omega, m}(\lm)\big)
			=
			\frac{1}{(m-1)!}\Tr\Bigg(
			\frac{\dd^{m-1}}{\dd\lm^{m-1}}
			\Big(\big(I -\omega \wh M(\lambda)\big)^{-1}\omega\wh M'(\lambda)
			\Big)\Bigg).
		\]
		\item
		For all $m \in \dN$ such that $m  > \frac{d-1}{2}$ and
		all $\lm\in\rho(\Opp)$ 	the resolvent power difference
		\[
			\wh D_{\omega, m}(\lm)
			\defeq ( \Opp -\lm )^{-m} - ( \Opfree-\lm )^{-m}
		\]
		belongs to the trace class, and its trace can be expressed as
        	\[
			\Tr\big(\wh D_{\omega, m}(\lm)\big) =
			\frac{1}{(m-1)!}\Tr\Bigg(
			\frac{\dd^{m-1}}{\dd\lm^{m-1}}
			\Big(\big(I -\omega \wh M(\lambda)\big)^{-1}
			\wh M(\lm)^{-1}\wh M'(\lambda)\Big)\Bigg).
		\]
	\end{myenum}
\end{theorem}

\medskip

\noindent
We remark that it is also implicitly shown that the derivatives of the operator-valued functions appearing
in the trace formulae exist in a suitable sense and that these derivatives belong to the trace class.

The main ingredients used in the proofs, which are given in Section~\ref{sec:proofs},
are Krein-type resolvent formulae,
properties of weak Schatten--von Neumann classes,
asymptotics of eigenvalues of the Laplace--Beltrami operator
on $\Sigma$, and elements of elliptic regularity theory.
We point out that for the proof of Theorem~\ref{thm2}\,(ii) an auxiliary trace formula
for the resolvent power differences of $\Opfree$ and $\OpN$ is derived in
Lemma~\ref{lem:aux}.  This trace formula is also of certain independent interest.
We also mention that a similar strategy of proof was employed in our previous
paper \cite{BLL13_LMS} where we proved trace formulae for generalized Robin Laplacians.

\section{Preliminaries}

This section consists of five subsections. In Subsection~\ref{ssec:Sp}
we recall the notion of weak Schatten--von Neumann classes
and their connection with the trace class,
and in Subsection~\ref{ssec:derivative} we collect certain formulae that involve
derivatives of holomorphic operator-valued functions.
Next, in Subsection~\ref{ssec:qbt} we recall the definitions of
quasi boundary triples and associated $\gamma$-fields and Weyl functions.
Krein's resolvent formulae and sufficient conditions for self-adjointness of
extensions are discussed in Subsection~\ref{ssec:Krein}.
Finally, in Subsection~\ref{ssec:qbt_pde} we introduce specific quasi boundary triples,
which are used to parameterize Schr\"odinger operators with singular interactions
from Definition~\ref{def:Ops}.

\subsection{$\sS_{p,\infty}$-classes and the trace mapping}
\label{ssec:Sp}

Let $\cH$ and $\cK$ be Hilbert spaces.
Denote by $\sS_\infty(\cH,\cK)$ the class of all compact operators
$K: \cH\arr\cK$.
Recall that, for $p>0$,
the \emph{weak Schatten--von Neumann ideal} $\sS_{p,\infty}(\cH,\cK)$ is defined by
\begin{equation*}\label{def_Sp}
	\sS_{p,\infty}(\cH,\cK) \defeq \Bigl\{K\in\sS_\infty(\cH,\cK): s_k(K)
	= \rmO\bigl(k^{-1/p}\bigr),\,k\to\infty\Bigr\}.
\end{equation*}
Often we just write $\sS_{p,\infty}$ instead of $\sS_{p,\infty}(\cH,\cK)$.
For $0< p' < p$  the inclusion
\begin{equation}\label{Sp_con_Spinf}
	\sS_{p,\infty} \subset \sS_{p',\infty}
\end{equation}
holds, and for $s,t >0$ one has
\begin{equation}\label{prod_Sp}
	\sS_{\frac{1}{s},\infty}\cdot\sS_{\frac{1}{t},\infty} =  \sS_{\frac{1}{s+t},\infty},
\end{equation}
where a product of operator ideals is defined as the set of all products.
We refer the reader to \cite[\S\S III.7 and III.14]{GK69} and \cite[Chapter~2]{S05}
for a detailed study of the classes $\sS_{p,\infty}$; see also \cite[Lemma~2.3]{BLL13_IEOT}.
If $K\in \sS_{p,\infty}$ with $p < 1$, then $K$ belongs to the trace class.
It is well known (see, e.g.\ \cite[\S III.8]{GK69}) that, for trace class operators
$K_1,K_2$, the operator $K_1 + K_2$ is also in the trace class, and that
\begin{equation}\label{trace1}
  \Tr(K_1 + K_2) = \Tr K_1 + \Tr K_2.
\end{equation}
Moreover, if $K_1\in\cB(\cH,\cK)$ and $K_2\in\cB(\cK,\cH)$ are such that
both products $K_1 K_2$ and $K_2 K_1$ are in the trace class, then
\begin{equation}\label{trace2}
  \Tr(K_1K_2) = \Tr( K_2 K_1).
\end{equation}
The next useful lemma is a special case of \cite[Lemma~4.7]{BLL13_IEOT} and is based
on the asymptotics of the eigenvalues of the Laplace--Beltrami operator.
For a smooth compact manifold $\Sigma$ we denote the usual $L^2$-based Sobolev
spaces by $H^r(\Sigma)$, $r\geq 0$.

\begin{lemma}\label{le.s_emb}
	Let $\Sigma$ be a $(d-1)$-dimensional
	compact $C^\infty$-manifold without boundary, let
	$\cK$ be a Hilbert space and let $K\in\cB(\cK,L^2(\Sigma))$
	with $\ran K \subset H^r(\Sigma)$, where $r > 0$.
	Then $K$ is compact and $K\in\sS_{\frac{d-1}{r},\infty}$.
\end{lemma}

\subsection{Derivatives of holomorphic operator-valued functions}
\label{ssec:derivative}

In the following we shall often use product rules for holomorphic operator-valued
functions.  Let $\cH_i$, $i = 1,\dots,4$, be Hilbert spaces, $\cU$ a domain in $\CC$
and let $A: \cU\to\cB(\cH_3,\cH_4)$, $B: \cU\to\cB(\cH_2,\cH_3)$,
$C: \cU\to\cB(\cH_1,\cH_2)$ be holomorphic operator-valued functions.  Then
for $\lm\in \cU$ we have
\begin{subequations}
\begin{align}
	\label{rule1}
	\frac{\dd^m}{\dd\lm^m}\bigl(A(\lm)B(\lm)\bigr)
	& = \sum_{ \substack{p+q = m\\ p,q \ge 0}} \binom{m}{p} A^{(p)}(\lm)B^{(q)}(\lm), \\
	\frac{\dd^m}{\dd\lm^m}\bigl(A(\lm)B(\lm)C(\lm)\bigr)
	&= \sum_{\substack{p+q+r = m\\  p,q,r \ge 0}}
	\frac{m!}{p!\,q!\,r!} A^{(p)}(\lm) B^{(q)}(\lm) C^{(r)}(\lm).
	\label{rule2}
\end{align}
\end{subequations}
If $A(\lm)^{-1}$ is invertible for every $\lm\in \cU$, then
relation \eqref{rule1} implies the following formula for the derivative
of the inverse,
\begin{equation}\label{rule3}
	\frac{\dd}{\dd\lm}\bigl(A(\lm)^{-1}\bigr) = -A(\lm)^{-1}A'(\lm)A(\lm)^{-1}.
\end{equation}
%

\subsection{Quasi boundary triples, Weyl functions and {\boldmath$\gamma$}-fields}
\label{ssec:qbt}

We begin this subsection by recalling
the abstract concept of \emph{quasi boundary triples} introduced in \cite{BL07}
as a generalization of the notion of (ordinary) \emph{boundary triples} \cite{Br76, Ko75}.
For the theory of ordinary boundary triples and associated Weyl functions the
reader may consult, e.g.\ \cite{BGP08, DM91, DM95}.
Recent developments on quasi boundary triples and their applications to PDEs
can be found in, e.g.\ \cite{BL12, BLL13_IEOT, BLL13_LMS, BR15}.
\begin{definition}\label{def:qbt}
	Let $S$ be a closed, densely defined, symmetric operator in a Hilbert
	space $(\cH, (\cdot,\cdot)_\cH)$.
	A triple $\{\cG,\Gamma_0,\Gamma_1\}$ is called a \emph{quasi boundary triple}
	for $S^*$ if $(\cG,(\cdot,\cdot)_\cG)$ is a Hilbert space, and for some
	linear operator $T\subset S^*$ with $\ov T  = S^*$ the following assumptions
	are satisfied:
	\begin{myenum}
		\item
		$\Gamma_0,\Gamma_1: \dom T\arr\cG$ are linear mappings,
		and the
		mapping $\Gamma \defeq \binom{\Gamma_0}{\Gamma_1}$
		has dense range  in $\cG\times\cG$;
		\item
		$A_0 \defeq T\upharpoonright\ker\Gamma_0$ is a self-adjoint operator
		in $\cH$;
		\item
		for all $f,g\in \dom T$ the \emph{abstract Green identity} holds:
		\begin{equation*}\label{green}
			(Tf,g)_{\cH}-(f,Tg)_{\cH}
			=
			(\Gamma_1 f,\Gamma_0 g)_{\cG}-(\Gamma_0 f,\Gamma_1 g)_{\cG}.
		\end{equation*}
	\end{myenum}
\end{definition}

\medskip

\noindent
Next, we recall the definitions of the $\gamma$-field and the Weyl
function associated with a quasi boundary triple $\{\cG, \Gamma_0,
\Gamma_1 \}$ for $S^*$. Note that the decomposition
\[
	\dom T = \dom A_0\,\dot +\,\ker(T-\lambda)
\]
holds for all $\lm\in\rho(A_0)$, so that $\Gamma_0\uhr\ker(T-\lambda)$
is injective for all $\lm\in\rho(A_0)$.
The (operator-valued) functions $\gamma$ and $M$ defined by
\begin{equation*}\label{gweyl}
	\gamma(\lambda) \defeq
	\bigl(\Gamma_0\upharpoonright\ker(T-\lambda)\bigr)^{-1}
	\quad\text{and}\quad
	M(\lambda) \defeq \Gamma_1\gamma(\lambda),\quad
	\lambda\in\rho(A_0),
\end{equation*}
are called the $\gamma$\emph{-field} and the \emph{Weyl function}
corresponding to the quasi boundary triple $\{\cG,\Gamma_0,\Gamma_1\}$.
The adjoint of $\gamma(\ov\lm)$ has the following representation:
\begin{equation}\label{eq:g*}
	\gamma(\ov\lm)^* = \Gamma_1(A_0 - \lm)^{-1},\qquad \lm\in\rho(A_0);
\end{equation}
see \cite[Proposition~2.6\,(ii)]{BL07} and also \cite[Proposition~7.5]{BL12}.
According to \cite[Proposition~2.6]{BL07} the operator-valued functions
$\lm\mapsto \gamma(\lm)$, $\lm\mapsto \gamma(\ov\lm)^*$
and $\lm\mapsto M(\lm)$ are holomorphic on $\rho(A_0)$.
Finally, we recall formulae for their derivatives:
for $k\in\dN$, $\varphi\in \ran\Gamma_0$ and $\lambda\in\rho(A_0)$ we have
\begin{subequations}\label{der_gamma_field_Weyl_function}
\begin{align}
	\gamma^{(k)}(\lm)\varphi &= k!(A_0 -\lm)^{-k}\gamma(\lm)\varphi,
	\label{der:g} \\[0.5ex]
	\frac{\dd^k}{\dd\lm^k}\bigl(\gamma(\ov\lm)\bigr)^*
	&= k!\,\gamma(\ov\lm)^*(A_0 -\lm)^{-k},
	\label{der:g*} \\[1ex]
	M^{(k)}(\lm)\varphi
	&= k!\,\gamma(\ov\lm)^* (A_0 -\lm)^{-(k-1)}\gamma(\lm)\varphi;
	\label{der:M}
\end{align}
\end{subequations}
see \cite[Lemma~2.4]{BLL13_LMS}.

\subsection{Self-adjoint extensions and abstract Krein's resolvent formula}\label{ssec:Krein}
\label{ssec:Krein}

In this subsection we parameterize subfamilies of self-adjoint extensions
via quasi boundary triples and provide a couple of useful Krein-type formulae
for resolvent differences of these extensions.

The following hypothesis will be useful in the following.
\begin{hypothesis}\label{hyp:qbt}
	Let $S$ be a closed, densely defined, symmetric operator in a Hilbert
	space $\cH$ and let $\{\cG,\Gamma_0,\Gamma_1\}$ be a quasi boundary triple for $S^*$
	such that $\ran\Gamma_0 = \cG$.
	Moreover, let $\gamma$ and $M$ be the associated $\gamma$-field and Weyl
	function, respectively.
\end{hypothesis}
We remark that the quasi boundary triple $\{\cG,\Gamma_0,\Gamma_1\}$
in Hypothesis~\ref{hyp:qbt} is also a generalized boundary triple in the sense
of~\cite{DM95}.  In this case the $\gamma$-field and the Weyl function
associated with $\{\cG,\Gamma_0,\Gamma_1\}$ are defined on the whole space $\cG$,
and the formulae~\eqref{der:g} and~\eqref{der:M} are valid for all $\varphi\in\cG$.

Next, we state a Krein-type formula for the resolvent difference of
$A_j \defeq T \uhr \ker\Gamma_j$, $j=0,1$.
\begin{proposition} {\rm\cite[Theorem~2.5]{BLL13_LMS}}
	Assume that Hypothesis~\ref{hyp:qbt} is satisfied and that $A_1$ is
	self-adjoint in $\cH$.  Then the formula
	\[
		(A_0 - \lm)^{-1} - (A_1 -\lm)^{-1} = \gamma(\lm) M(\lm)^{-1} \gamma(\ov\lm)^*
	\]
	holds for all $\lm\in\rho(A_0)\cap\rho(A_1)$.
	\label{prop:Krein1}
\end{proposition}
In the next proposition, we formulate a sufficient condition for
self-adjointness of the extension of $A$ defined by
\[
	\AB \defeq T\uhr \ker(B\Gamma_1 - \Gamma_0),
\]
and provide a Krein-type formula for the resolvent difference of $\AB$ and $A_0$.

\begin{proposition} {\rm\cite[Theorem~2.6]{BLL13_LMS}	}
	\label{prop:Krein2}
	Assume that Hypothesis~\ref{hyp:qbt} is satisfied, that $M(\lm_0)\in \sS_\infty(\cG)$
	for some $\lm_0\in\rho(A_0)$,
	and that $B\in \cB(\cG)$ is self-adjoint in $\cG$.
	Then the extension $\AB$ of $A$ is self-adjoint in $\cH$,
	and the formula
	\[
		(\AB - \lm)^{-1} - (A_0 -\lm)^{-1}
		= \gamma(\lm)\big (I - BM(\lm)\big)^{-1}B\gamma(\ov\lm)^*
	\]
	holds for all $\lm\in\rho(\AB)\cap\rho(A_0)$.
	In this formula the middle term satisfies
	\[
		\big (I - BM(\lm)\big)^{-1}\in \cB(\cG)
	\]
	for all $\lm\in\rho(\AB)\cap\rho(A_0)$.
\end{proposition}

\subsection{Quasi boundary triples for coupled problems}
\label{ssec:qbt_pde}

We recall particular quasi boundary triples, which are used to parameterize
the self-adjoint operators from Definition~\ref{def:Ops}.
Furthermore, we reformulate some of the abstract statements from
Subsections~\ref{ssec:qbt} and~\ref{ssec:Krein}
for these quasi boundary triples.

First, we introduce the subspace $H^{3/2}_\Delta(\dR^d\setminus\Sigma)$ of $L^2(\dR^d)$ by
\begin{equation*}\label{eq:H32}
	H^{3/2}_\Delta(\dR^d\setminus\Sigma) \defeq H^{3/2}_\Delta(\Omega_+)\oplus H^{3/2}_\Delta(\Omega_-),
\end{equation*}
where $H^{3/2}_\Delta(\Omega_\pm)$ are as in~\eqref{eq:H32pm}.
Further, to shorten the notations, we also define the
\emph{jump of the normal derivative}
by $[ \partial_\nu u ]_\Sigma \defeq \partial_{\nu_+}u_+|_\Sigma + \partial_{\nu_-}u_-|_\Sigma$
for $u\in H^{3/2}_\Delta(\dR^d\setminus\Sigma)$.
Following the lines of \cite[Section~3]{BLL13_AHP},
we define the operators $\wt T$ and $\wh T$ in $L^2(\dR^d)$ by
\begin{alignat*}{2}
	\wt T u &\defeq (-\Delta u_+)\oplus (-\Delta u_-), \qquad &
	\dom \wt T &\defeq \bigl\{u\in H^{3/2}_\Delta(\dR^d\setminus\Sigma): [u]_\Sigma=0\bigr\},
	\\[0.5ex]
	\wh T u &\defeq (-\Delta u_+)\oplus (-\Delta u_-), \qquad &
	\dom \wh T &\defeq \bigl\{u\in H^{3/2}_\Delta(\dR^d\setminus\Sigma): [\partial_\nu u]_\Sigma=0\bigr\},
\end{alignat*}
and their restrictions $\wt S$ and $\wh S$ by
\begin{align*}
	\wt S  &\defeq  \wt T\uhr \bigl\{ u\in H^{3/2}_\Delta(\dR^d\setminus\Sigma):
	u_\pm|_\Sigma = 0, ~[\partial_\nu u]_\Sigma = 0\bigr\},
	\\[1ex]
	\wh S  &\defeq \wh T\uhr \bigl\{ u\in H^{3/2}_\Delta(\dR^d\setminus\Sigma):
	\partial_{\nu_\pm} u_\pm|_\Sigma = 0,~[u]_\Sigma = 0\bigr\}.
\end{align*}
It can be  verified
that $\wt S$ (respectively, $\wh S$) is the restriction of $\Opfree$ to 	
functions, whose Dirichlet trace (respectively, Neumann trace) vanishes on
$\Sigma$. In particular, as a consequence of this identification we
arrive at the inclusions $\dom\wt S, \dom\wh S \subset H^2(\dR^d)$.
It can also be shown that
the operators $\wt S$ and $\wh S$ are closed, densely defined, and
symmetric in $L^2(\dR^d)$ and that the
closures of $\wt T$ and $\wh T$ coincide with $\wt S^*$
and $\wh S^*$, respectively.  Furthermore, we define the boundary mappings by
\begin{alignat}{3}
	\wt\Gamma_0,\wt\Gamma_1 &: \dom\wt T\to L^2(\Sigma), \qquad &
	\wt\Gamma_0 u &\defeq [\partial_\nu u]_\Sigma, \qquad &
	\wt\Gamma_1 u &\defeq u|_\Sigma,
	\label{eq:wtG01} \\[1ex]
	\wh\Gamma_0,\wh\Gamma_1 &: \dom\wh T\to L^2(\Sigma), \qquad &
	\wh\Gamma_0 u &\defeq \partial_{\nu_+}u_+|_{\Sigma}, \qquad &
	\wh\Gamma_1 u &\defeq [u]_\Sigma.
	\label{eq:whG01}
\end{alignat}	
The identities
\[
	\Opfree = \wt T\uhr \ker\wt\Gamma_0 = \wh T\uhr \ker\wh\Gamma_1
	\qquad\text{and}\qquad
	\OpN = \wh T\uhr\ker\wh\Gamma_0
\]
can be checked in a straightforward way.
According to \cite[Proposition~3.2\,(i)]{BLL13_AHP} the
triple $\wt \Pi \defeq \{L^2(\Sigma),\wt\Gamma_0,\wt\Gamma_1\}$
is a quasi boundary triple for $\wt S^*$, and by \cite[Proposition~3.8\,(i)]{BLL13_AHP}
the triple $\wh \Pi \defeq \{L^2(\Sigma),\wh\Gamma_0,\wh\Gamma_1\}$
is a quasi boundary triple for $\wh S^*$.

\begin{definition}\label{def:gM}
	Let $\wt \gamma$, $\wt M$ and $\wh\gamma$, $\wh M$ be
	the  $\gamma$-fields and the Weyl functions
	of the quasi boundary triples $\wt\Pi$ and $\wh \Pi$, respectively.
\end{definition}
\begin{remark}
	The definitions of the operator-valued functions $\wt M$ and $\wh M$
	as Neumann-to-Dirichlet maps in \eqref{def:M} and
	as Weyl functions of the quasi boundary triples
	$\wt\Pi$ and $\wh\Pi$ are equivalent;
	see \cite[Propositions~3.2\,(iii) and 3.8\,(iii)]{BLL13_AHP}.
\end{remark}

\begin{remark}
	According to \cite[Propositions~3.2\,(ii) and 3.8\,(ii)]{BLL13_AHP},
	for any $\varphi \in L^2(\Sigma)$ both \emph{transmission boundary value problems}
	\[
		\left\{\begin{alignedat}{2}
		-\Delta u  &= \lm u \qquad && \text{in}\;\; \dR^d\setminus\Sigma, \\
		[u]_\Sigma &= 0 \qquad && \text{on}\;\; \Sigma, \\
		[\partial_\nu u]_\Sigma &= \varphi \qquad && \text{on}\;\; \Sigma,
	\end{alignedat}\right.
	\qquad\qquad
	\left\{
	\begin{alignedat}{2}
		-\Delta u  &= \lm u \qquad && \text{in}\;\; \dR^d\setminus\Sigma, \\
		\partial_{\nu_+}u_{\rm +}|_\Sigma &= \varphi \qquad && \text{on}\;\; \Sigma, \\
		\partial_{\nu_-}u_-|_\Sigma &= -\varphi \qquad && \text{on}\;\; \Sigma,
	\end{alignedat}\right.
	\]
	have unique solutions
	$\wt u(\varphi),\wh u(\varphi) \in H^{3/2}_\Delta(\dR^d\setminus\Sigma)$.
	Moreover, the operator-valued functions
	$\wt\gamma$ and $\wh\gamma$
	satisfy $\wt\gamma(\lm)\varphi = \wt u_\lm(\varphi)$ and
	$\wh\gamma(\lm)\varphi = \wh u_\lm(\varphi)$ for $\varphi \in L^2(\Sigma)$
	and $\lambda\in\dC\setminus\dR_+$.
\end{remark}
Thanks to \eqref{eq:g*} the adjoints of $\wt\gamma(\ov\lm)$ and $\wh\gamma(\ov\lm)$
can be expressed as
\begin{equation}\label{eq:gamma_adjoints}
	\wt\gamma(\ov\lm)^* = \wt\Gamma_1(\Opfree - \lm)^{-1}
	\qquad\text{and}\qquad
	\wh\gamma(\ov\lm)^* = \wh\Gamma_1(\OpN - \lm)^{-1}
\end{equation}
for $\lm\in\dC\setminus\dR_+$.
We also remark that, by \cite[Propositions~3.2\,(iii) and 3.8\,(iii)]{BLL13_AHP},
we have
\begin{equation}\label{eq:Weyl_ranges}
	\ran\wt M(\lm) = \ran \wh M(\lm) = H^1(\Sigma),\qquad 	
	\lm\in\dC\setminus\dR_+.
\end{equation}
According to Proposition~\ref{prop:Krein1}, the formula
\begin{equation}\label{eq:Krein1}
	(\OpN -\lm)^{-1} - (\Opfree -\lm)^{-1}
	= \wh\gamma(\lm)\wh M(\lm)^{-1}\wh\gamma(\ov\lm)^*,
\end{equation}
holds for all $\lm\in\dC\setminus\dR_+$.
Since the operators of multiplication with $\alpha$ and $\omega$ are bounded
and self-adjoint in $L^2(\Sigma)$,
by Proposition~\ref{prop:Krein2} the
extensions
\[
	\wt T\uhr \ker(\alpha\wt\Gamma_1 - \wt \Gamma_0)
	\qquad\text{and}\qquad
	\wh T\uhr\ker(\omega\wh\Gamma_1 - \wh \Gamma_0)
\]
are self-adjoint in $L^2(\dR^d)$.  In a way similar to \cite{BLL13_AHP},
one can check that these restrictions coincide with $\Op$ and $\Opp$, respectively.
Moreover, by Proposition~\ref{prop:Krein2}, the formulae
\begin{subequations}
\begin{align}      	
	(\Op-\lm)^{-1}-(\Opfree-\lm)^{-1}
	&= \wt\gamma(\lm)\,\bigl(I - \alpha\wt M(\lm)\bigr)^{-1}\alpha\,\wt\gamma(\ov\lm)^*,
	\label{eq:Krein2} \\[0.5ex]
	(\Opp-\lm)^{-1}-(\OpN-\lm )^{-1}
	&= \wh\gamma(\lm)\,\bigl(I - \omega\wh M(\lm)\bigr)^{-1}\omega\,\wh\gamma(\ov\lm)^*,
	\label{eq:Krein3}
\end{align}	
\end{subequations}
hold for all $\lm\in\rho(\Op)$ and all $\lm\in\rho(\Opp)$,
respectively. In these formulae the middle terms on the right-hand sides satisfy
\begin{equation}\label{eq:bnd_inv}
	\bigl(I-\alpha\wt M(\lm)\bigr)^{-1}, \bigl(I-\omega \wh M(\lm)\bigr)^{-1}\in \cB(L^2(\Sigma))
\end{equation}	
for $\lm$ in the respective resolvent sets.

\section{Proofs of the main results}\label{sec:proofs}

In this section we prove the main results of the paper: the trace formulae
for the Schr\"odinger operators with singular interactions.
Theorems~\ref{thm1} and \ref{thm2} are proved in Subsections~\ref{ssec:proof_thm1}
and \ref{ssec:proof_thm2}, respectively.
Throughout this section we use the notations
$R(\lm) \defeq (\Opfree - \lm)^{-1}$ and
$R_{\rm N}(\lm) \defeq (\OpN - \lm)^{-1}$.

\subsection{Proof of Theorem~\ref{thm1}}
\label{ssec:proof_thm1}

To prove Theorem~\ref{thm1} we need an auxiliary lemma.

\begin{lemma}\label{lem:der1}
	Let the $\gamma$-field $\wt\gamma$ and the Weyl function $\wt M$
	be as in Definition~\ref{def:gM}.
	Then for every $\lm \in \dC\setminus\dR_+$ and every $k\in\dN_0$ the
	following relations hold:
	\begin{myenum}
	\item $\wt\gamma^{(k)}(\lm),
	\frac{\dd^k}{\dd\lm^k}\wt\gamma(\ov\lm)^*\in \sS_{\frac{d-1}{2k+3/2},\infty}$;
	\item $\wt M^{(k)}(\lambda) \in \sS_{\frac{d-1}{2k+1},\infty}$.
	\end{myenum}
\end{lemma}

\begin{proof}
	{\rm (i)}
	Let $\lm\in\dC\setminus\dR_+$ and $k\in\dN_0$.  First, we observe that
	$\ran(R(\lm)^k) \subset H^{2k}(\dR^d)$.
	By the trace theorem we have
	$u|_{\Sigma} \in H^{s-1/2}(\Sigma)$ for every $u\in H^s(\dR^d)$
	with $s > 1/2$.
	Hence, we obtain from~\eqref{eq:gamma_adjoints} that
	\[
		\ran\bigl(
		\wt\gamma(\ov\lambda)^*R(\lm)^k\bigr)
		\subset H^{2k+3/2}( \Sigma ).
	\]
	Thus Lemma~\ref{le.s_emb} with $\cK = L^2(\dR^d)$ and $r = 2k+3/2$ implies that
	\begin{equation}\label{gres_SvN}
		\wt\gamma(\ov\lambda)^*R(\lm)^k \in \sS_{\frac{d-1}{2k+3/2},\infty}.
	\end{equation}
	By taking the adjoint in~\eqref{gres_SvN} and replacing $\lm$ by $\ov\lm$ we obtain
	\begin{equation}\label{resg_SvN}
		R(\lm)^k
		\wt\gamma(\lm)\in
		\sS_{\frac{d-1}{2k+3/2},\infty}.
	\end{equation}
	From \eqref{der:g}, \eqref{der:g*}, \eqref{gres_SvN} and \eqref{resg_SvN}
	we now obtain
	$\wt\gamma^{(k)}(\lm), \frac{\dd^k}{\dd\lm^k}\wt\gamma(\ov\lm)^*\in \sS_{\frac{d-1}{2k+3/2},\infty}$.
	
	{\rm (ii)}
	For $k=0$ we observe that by~\eqref{eq:Weyl_ranges} we have
	$\ran \wt M(\lm) = H^1(\Sigma)$.
	Therefore, Lemma~\ref{le.s_emb} with $\cK = L^2(\Sigma)$ and $r = 1$
	implies that $\wt M(\lm)\in\sS_{d-1,\infty}$.
	For $k \ge 1$ we derive from~\eqref{der:M} that
	\[
		\wt M^{(k)}(\lm)
		=
		k!\,\wt\gamma(\ov\lambda)^*R(\lm)^{k-1}
		\wt\gamma(\lm)\in
		\sS_{\frac{d-1}{2(k-1)+3/2},\infty}
		\cdot\sS_{\frac{d-1}{3/2},\infty}
		 =
		\sS_{\frac{d-1}{2k+1},\infty},
	\]
	where we applied \eqref{gres_SvN}, \eqref{resg_SvN}
	and \eqref{prod_Sp}.
\end{proof}

\label{ssec:thm1}

\begin{proof}[Proof of Theorem~\ref{thm1}]
	In order to shorten notation and to avoid
	the distinction of several cases, we set
	\[
		\sA_r
		\defeq
		\begin{cases}
		    \sS_{\frac{d-1}{r},\infty}\bigl(L^2(\Sigma)\bigr) & \text{if}\; r>0, \\[1ex]
	   		\cB\bigl(L^2(\Sigma)\bigr) & \text{if}\; r=0.
		\end{cases}
	\]
	It follows from~\eqref{prod_Sp}
	and the fact that $\frS_{p,\infty}(L^2(\Sigma))$ is an
	ideal in $\cB(L^2(\Sigma))$ for $p>0$ that
	\begin{equation}\label{prodAr}
		\frA_{r_1}\cdot\frA_{r_2} = \frA_{r_1+r_2}, \qquad r_1,r_2 \ge0.
	\end{equation}
	The remainder of the proof is divided into two steps.
	\medskip

	\noindent
	\textit{Step 1.}
	Let $\alpha\in L^\infty(\Sigma;\dR)$ and set
	\[
		\wt T(\lm) \defeq \bigl(I-\alpha\wt M(\lm)\bigr)^{-1}, \qquad \lm\in\rho(\Op),
	\]
	where $\wt T(\lm)\in\cB(L^2(\Sigma))$ by \eqref{eq:bnd_inv}.
	Next, we show that
	\begin{equation}\label{TkinAk}
		 \wt T^{(k)}(\lambda) \in \frA_{2k+1}, \qquad k\in\NN,
	\end{equation}
	by induction.  Relation~\eqref{rule3} implies that
	\begin{equation}\label{Tder}
		\wt T'(\lm) = \wt T(\lm)\alpha \wt M'(\lm) \wt T(\lm),
	\end{equation}
	which is in $\frA_3$ by Lemma~\ref{lem:der1}\,(ii).
	Let $m\in\NN$ and assume that~\eqref{TkinAk} is true
	for every $k=1,\dots,m$, which implies, in particular, that
	\begin{equation}\label{TkinA2k}
	 	\wt T^{(k)}(\lambda) \in \frA_{2k}, \qquad k=0,\dots,m.
	\end{equation}
	Then
	\begin{align*}
		\wt T^{(m+1)}(\lm)
		&=
		\frac{\dd^m}{\dd\lm^m}\Bigl(\wt T(\lm)\alpha \wt M'(\lm)\wt T(\lm)\Bigr)
		\\[1ex]
		&=
		\sum_{\substack{p+q+r=m\\ p,q,r\ge0}}
		\frac{m!}{p!\,q!\,r!} \wt T^{(p)}(\lm)\alpha \wt M^{(q+1)}(\lm)\wt T^{(r)}(\lm)
	\end{align*}
	by~\eqref{Tder} and \eqref{rule2}.
	Relation \eqref{TkinA2k}, the boundedness of $\alpha$,
	Lemma~\ref{lem:der1}\,(ii) and \eqref{prodAr} imply that
	\[
		\wt T^{(p)}(\lm)\alpha \wt M^{(q+1)}(\lm)\wt T^{(r)}(\lm)
		\in \frA_{2p}\cdot\frA_{2(q+1)+1}\cdot\frA_{2r}
		=
		\frA_{2(m+1)+1},
	\]
	since $p+q+r=m$. This shows \eqref{TkinAk} for $k=m+1$ and hence,
	by induction, for all $k\in\NN$. Since $\wt T(\lm)\in\cB(L^2(\Sigma))$,
	we have, in particular,
	\begin{equation}\label{TkinAk_weak}
		\wt T^{(k)}(\lm) \in \frA_{2k}, \qquad k\in\NN_0,\;	
		\lm\in\rho(\Op).
	\end{equation}
	%

	\noindent
	\textit{Step 2.}
	By taking derivatives we obtain from \eqref{eq:Krein2} that, for $m\in\NN$,
	\begin{align}
		\notag
		(m-1)!\, \wt D_{\alpha, m}(\lm)
		& = \frac{\dd^{m-1}}{\dd\lm^{m-1}}\big(\wt D_{\alpha, 1}(\lm)\big) =
		\frac{\dd^{m-1}}{\dd\lm^{m-1}}\Bigl(\wt\gamma(\lm)\wt T(\lm)\alpha\wt\gamma(\ov\lm)^*\Bigr) \\
		& =
		\sum_{\substack{p+q+r=m-1\\ p,q,r\ge0}}\frac{(m-1)!}{p!\,q!\,r!}
		\wt\gamma^{(p)}(\lm)\wt T^{(q)}(\lm)\alpha\frac{\dd^r}{\dd\lm^r}
		\wt\gamma(\ov\lm)^*.
		\label{sum594}
	\end{align}
	By Lemma~\ref{lem:der1}\,(i) and~\eqref{TkinAk_weak},
	each term in the sum satisfies
	\begin{equation}\label{terms_in_Sp2}
		\wt\gamma^{(p)}(\lm)\wt T^{(q)}(\lm)\alpha
		\frac{\dd^r}{\dd\lm^r}\wt\gamma(\ov\lm)^*
		\in \frA_{2p+3/2}\cdot\frA_{2q}\cdot\frA_{2r+3/2}
		= \frA_{2m +1} = \frS_{\frac{d-1}{2m+1},\infty}.
	\end{equation}
	If $m\in\dN$ is such that  $m > \frac{d-2}{2}$, then $\frac{d-1}{2m+1} < 1$ and,
	by \eqref{Sp_con_Spinf} and \eqref{terms_in_Sp2}, all terms in the sum
	in~\eqref{sum594}
	are trace class operators, and the same is true if we change the order in the 		
	product in \eqref{terms_in_Sp2}.
	Hence, we can apply the trace to the expression in~\eqref{sum594} and
	use \eqref{trace1}, \eqref{trace2} and \eqref{der:M} to obtain
	\begin{align*}
	 	& (m-1)!\,\Tr\big(\wt D_{\alpha, m}(\lm)\big)
		=
		\Tr\Biggl(\,\sum_{\substack{p+q+r=m-1 \\ p,q,r\ge0}}\frac{(m-1)!}{p!\,q!\,r!}
		\wt\gamma^{(p)}(\lm) \wt T^{(q)}(\lm)\alpha \frac{\dd^r}{\dd\lm^r}\wt\gamma(\ov\lm)^*\Biggr)\\
		& =
		\sum_{\substack{p+q+r=m-1 \\ p,q,r\ge0}}
		\frac{(m-1)!}{p!\,q!\,r!}
		\Tr\Bigl(\wt\gamma^{(p)}(\lm)\wt T^{(q)}(\lm)
		\alpha\frac{\dd^r}{\dd\lm^r}\wt\gamma(\ov\lm)^*\Bigr) \\
		&=
		\sum_{\substack{p+q+r=m-1 \\ p,q,r\ge0}}
		\frac{(m-1)!}{p!\,q!\,r!}
		\Tr\biggl(\wt T^{(q)}(\lm)\alpha\Bigl(\frac{\dd^r}{\dd\lm^r}
		\wt\gamma(\ov\lm)^*\Bigr)	\wt\gamma^{(p)}(\lambda)\biggr) \\
		& =
		\Tr\Biggl(\, \sum_{\substack{p+q+r=m-1 \\ p,q,r\ge0}}
		\frac{(m-1)!}{p!\,q!\,r!}
		\wt T^{(q)}(\lm)\alpha
		\Bigl(\frac{\dd^r}{\dd\lm^r}\wt\gamma(\ov\lambda)^*\Bigr)
		\wt\gamma^{(p)}(\lm)
		\Biggr) \\
		& =
		\Tr\biggl(
		\frac{\dd^{m-1}}{\dd\lm^{m-1}}
		\Bigl(\wt T(\lambda)\alpha\wt\gamma(\ov\lm)^*\wt\gamma(\lm)
		\Bigr)  \biggr)
		=
		\Tr\biggl(\frac{\dd^{m-1}}{\dd\lm^{m-1}}
		\Bigl(\wt T(\lm)\alpha\wt M'(\lm)\Bigr)\biggr),
	\end{align*}
	which finishes the proof.
\end{proof}

\subsection{Proof of Theorem~\ref{thm2}}
\label{ssec:proof_thm2}

First, we need three preparatory lemmas.
The proof of the first of them is completely analogous to the proof of
Lemma~\ref{lem:der1} and is therefore omitted.	

\begin{lemma}\label{lem:der2}
	Let the $\gamma$-field $\wh\gamma$ and the Weyl function $\wh M$ be as in
	Definition~\ref{def:gM}.
	Then for every $\lm \in \dC\setminus\dR_+$ and every $k\in\dN_0$
	the following relations hold:
	\begin{myenum}
		\item
		$\wh\gamma^{(k)}(\lm),
		\frac{\dd^k}{\dd\lm^k}\wh\gamma(\ov\lambda)^*\in
		\sS_{\frac{d-1}{2k+3/2},\infty}$;
		\item $\wh M^{(k)}(\lambda) \in \sS_{\frac{d-1}{2k+1},\infty}$.
	\end{myenum}
\end{lemma}

\begin{lemma}
\label{lem:mapping}
	Let the $\gamma$-field $\wh\gamma$ and the Weyl function $\wh M$ be as in
	Definition~\ref{def:gM}.
	Then for all $s\ge 0$, and all $\lm\in\dC\setminus\dR_+$ the following
	statements hold:
	\begin{myenum}
		\item $\ran\bigl(\wh\gamma(\ov\lm)^*\uhr  H^s(\dR^d)\bigr)
		\subset H^{s+\frac32}(\Sigma)$;
		\item $\ran\bigl(\wh M(\lm)\uhr H^s(\Sigma)\bigr) = H^{s+1}(\Sigma)$.
	\end{myenum}
\end{lemma}

\begin{proof}
(i)
According to \eqref{eq:gamma_adjoints} we have
\begin{equation*}
  \wh\gamma(\ov\lm)^* = \wh\Gamma_1 R_{\rm N}(\lm).
\end{equation*}
Employing  the regularity shift property \cite[Theorem~4.20]{McL}
and the trace theorem \cite[Theorem~3.37]{McL} we conclude that
\begin{equation*}
	\ran \bigl(\wh\gamma(\ov\lambda)^* \uhr H^s(\dR^d)\bigr) \subset
	 H^{s + \frac32}(\Sigma)
\end{equation*}
holds for all $s\geq 0$.

(ii)
Define the space $H^s(\dR^d\setminus\Sigma) \defeq H^s(\Omega_+)\oplus H^s(\Omega_-)$.
It follows from the decomposition
$\dom \wh T=\dom \OpN\dotplus\ker(\wh T-\lambda)$, $\lambda\in\dC\setminus\dR_+$,
and the properties of the Neumann trace \cite[\S2.7.3]{LM72-I}
that the restriction of the mapping $\wh\Gamma_0$ to
\begin{equation*}
  \ker(\wh T-\lambda)\cap H^{s+\frac32}(\dR^d\setminus\Sigma)
\end{equation*}
is a bijection onto $H^s(\Sigma)$ for $s \ge 0$.
This, together with the definition of the $\gamma$-field, implies that
\begin{equation*}
  \ran \bigl(\wh\gamma(\lm)\uhr H^s(\Sigma)\bigr)
  =  \ker(\wh T-\lambda)\cap
   H^{s+\frac32}(\dR^d\setminus\Sigma)	
  \subset  H^{s+\frac32}(\dR^d\setminus\Sigma).
\end{equation*}
Hence, it follows from the definition of $\wh M(\lambda)$,
the definition of $\wh\Gamma_1$ in \eqref{eq:whG01} and the trace theorem that
\[
	\ran(\wh M(\lm)\uhr H^s(\Sigma))
	\subset H^{s+1}(\Sigma).
\]
To verify the opposite inclusion, let $\psi\in H^{s+1}(\Sigma)$.
The decomposition $\dom \wh T = \dom \Opfree\dotplus\ker(\wh T-\lm)$,
$\lm\in\dC\setminus\dR_+$ implies that there exists a function
$f_\lm\in\ker(\wh T-\lambda)\cap H^{s+\frac32}(\dR^d\setminus\Sigma)$
such that $\wh \Gamma_1 f_\lambda = \psi$.  Thus,
\begin{equation*}
  \wh \Gamma_0 f_\lambda = \varphi\in H^s(\Sigma)
  \qquad\text{and}\qquad \wh M(\lambda)\varphi=\psi,
\end{equation*}
that is, $H^{s+1}(\Sigma)\subset\ran\bigl(\wh M(\lambda)\uhr H^s(\Sigma)\bigr)$,
and the assertion is shown.
\end{proof}

\begin{lemma}\label{lem:aux}
	Let the self-adjoint operators $\Opfree$ and $\OpN$ be as in Definition~\ref{def:Ops},
	and let the operator-valued function $\wh M$ be as in~\eqref{def:M}.
	Then for all $m \in \dN$ such that $m  > \frac{d-1}{2}$ and
	all $\lm\in\dC\setminus\dR_+$ the resolvent power difference
	\[
		\wh D_m( \lm ) \defeq ( \OpN -\lm )^{-m} - ( \Opfree - \lm )^{-m}
	\]
	belongs to the trace class, and its trace can be expressed as
	\[
		\Tr\bigl(\wh D_m(\lm)\bigr)
		= \frac{1}{(m-1)!}\Tr\biggl(\frac{\dd^{m-1}}{\dd\lm^{m-1}}
		\bigl(\wh M(\lm)^{-1}\wh M'(\lm)\bigr)\biggr).
	\]
\end{lemma}

\begin{proof}
	The proof is divided into three steps.
	\\[0.5ex]
	\textit{Step 1.}
	Let us introduce the operator-valued function
	\[
		S(\lm) \defeq \wh M(\lm)^{-1}\wh \gamma(\ov\lm)^*, \qquad \lm\in\dC\setminus\dR_+.
	\]
	Note that the product is well defined since by Lemma~\ref{lem:mapping}\,(i)
	\[
		\ran\bigl(\wh\gamma(\ov\lambda)^*\bigr) \subset H^1(\Sigma)
		= \dom\bigl(\wh M(\lm)^{-1}\bigr).
	\]
	The closed graph theorem implies that
	$S(\lm)\in \cB(L^2(\dR^d), L^2(\Sigma))$ for all $\lambda\in\dC\setminus\dR_+$.
	Next we prove the following smoothing property
	for the derivatives of $S$:
	\begin{equation}\label{Sk_smoothing}
		\ran\bigl(S^{(k)}(\lambda)\uhr H^s(\dR^d)\bigr)
		\subset H^{s + 2k + 1/2}(\Sigma),
		\qquad s\ge 0,\, k\in\dN_0,
	\end{equation}
	by induction.
	Since, by Lemma~\ref{lem:mapping}\,(i),
	$\wh\gamma(\ov\lambda)^*$ maps $H^s(\dR^d)$ into $H^{s+3/2}(\Sigma)$
	for all $s \ge 0$
	and $\wh M(\lm)^{-1}$ maps $H^{s+3/2}(\Sigma)$ into $H^{s+1/2}(\Sigma)$
	by Lemma~\ref{lem:mapping}\,(ii), relation~\eqref{Sk_smoothing} is true for $k=0$.
	Now let $l \in \dN_0$ and assume that \eqref{Sk_smoothing} is true for every
	$k=0,1,\dots,l$.
	It follows from \eqref{rule1}, \eqref{rule3}, \eqref{der:g*}, \eqref{der:M}
	and \eqref{eq:Krein1} that for $\lm\in\dC\setminus\dR_+$,
	\begin{align*}
		S'(\lm)
	  	&=
		\frac{\dd}{\dd\lm}
		\big(\wh M(\lm)^{-1}\big)\wh \gamma(\ov\lambda)^*
		+ \wh M(\lm)^{-1}\frac{\dd}{\dd\lambda}\wh\gamma(\ov\lm)^*  \\[0.5ex]
		&= -\wh M(\lm)^{-1}\wh M'(\lm)\wh M(\lm)^{-1}\wh \gamma(\ov\lambda)^* +
		\wh M(\lm)^{-1}\wh\gamma(\ov\lm)^*R_{\rm N}(\lm) \\[0.5ex]
		&= -\wh M(\lm)^{-1}\wh\gamma(\ov\lm)^*\wh\gamma(\lm)\wh M(\lm)^{-1}\wh\gamma(\ov\lm)^*
		+ \wh M(\lm)^{-1}\wh\gamma(\ov\lm)^*R_{\rm N}(\lm) \\[0.5ex]
		&= S(\lm)\big[R_{\rm N}(\lm)-
		\wh\gamma(\lm)\wh M(\lm)^{-1}\wh\gamma(\ov\lm)^*\big]
		= S(\lm)R(\lm).
	\end{align*}
	Hence, with the help of \eqref{rule1} we obtain
	\begin{equation}\label{eq:S(l)}
	\begin{aligned}
		S^{(l+1)}(\lambda)
		&=
		\frac{\dd^l}{\dd\lambda^l}\Bigl(
		S(\lm)R(\lm)\Bigr)
		=
		\sum_{\substack{p+q=l \\ p,q\ge0}}
		\binom{l}{p}	S^{(p)}(\lambda) R^{(q)}(\lm)
		\\[0.5ex]
		&=
		\sum_{\substack{p+q=l \\ p,q\ge0}}
		\frac{l!}{p!}S^{(p)}(\lambda)R(\lm)^{q+1}.
	\end{aligned}
	\end{equation}
	Using the induction hypothesis, formula \eqref{eq:S(l)} and
	smoothing properties of $R(\lm)$, we deduce that, for $p,q\geq 0$, $p+q=l$,
	\begin{align*}
		\ran (S^{(p)}(\lm)R(\lm)^{q+1} \uhr
		H^s(\dR^d)) & \subset
		\ran (S^{(p)}(\lm)\uhr H^{s+2(q+1)}(\dR^d)) \\[0.5ex]
		& \subset H^{s+2(p+q+1) +1/2}(\Sigma)
		=
		H^{s+2(l+1) + 1/2}(\Sigma),
	\end{align*}
	which shows \eqref{Sk_smoothing} for $k = l + 1$ and hence, by induction,
	for all $k\in\dN_0$.
    Therefore, an application of Lemma~\ref{le.s_emb}
    with $\cK = L^2(\Sigma)$ and  $r = 2k+1/2$ yields that
	\begin{equation}\label{Sk_in_Sp}
		S^{(k)}(\lm) \in
		\sS_{\frac{d - 1}{2k+1/2},\infty},
		\qquad k\in\dN_0,\,\lm\in\dC\setminus\dR_+.
	\end{equation}

	\noindent
	\textit{Step 2.}
	Using Krein's formula in \eqref{eq:Krein1}
	and \eqref{rule1} we obtain that,
	for $m\in\dN$ and $\lambda\in\dC\setminus\dR_+$,
	\begin{align}
	  \wh D_m(\lm)
	  &= \frac{1}{(m-1)!}\cdot\frac{\dd^{m-1}}{\dd\lambda^{m-1}}\bigl(\wh D_1(\lm)\bigr) =
		\frac{1}{(m-1)!}\cdot\frac{\dd^{m-1}}{\dd\lambda^{m-1}}\bigl(\wh\gamma(\lambda)S(\lambda)\bigr)\notag\\
	  & = \frac{1}{(m-1)!}\sum_{\substack{p+q=m-1 \\ p,q\ge0}}  \binom{m-1}{p} \wh\gamma^{(p)}(\lambda) S^{(q)}(\lambda).
	    \label{sum578}
	\end{align}
	By Lemma~\ref{lem:der2}\,(i), \eqref{Sk_in_Sp} and \eqref{prod_Sp} we have
	\begin{equation}\label{terms_in_Sp1}
		\wh\gamma^{(p)}(\lambda)S^{(q)}(\lambda)
		  \in \frS_{\frac{d-1}{2p+3/2},\infty}\cdot\frS_{\frac{d-1}{2q+1/2},\infty}
		  = \frS_{\frac{d-1}{2(p+q)+2},\infty} = \frS_{\frac{d-1}{2m},\infty}
	\end{equation}
	for $p,q$ with $p+q=m-1$.
	\medskip

	\noindent
	\textit{Step 3.}
	If $m>\tfrac{d-1}{2}$, then $\frac{d-1}{2m}<1$ and,
	by \eqref{terms_in_Sp1},
	each term in the sum in \eqref{sum578} is a trace class
	operator and, by a similar argument,
	also $S^{(q)}(\lambda)\wh\gamma^{(p)}(\lambda)$.
	Hence, the resolvent power difference $\wh D_m(\lm)$
	is a trace class operator, and
	we can apply the trace to \eqref{sum578}
	and use \eqref{trace1}, \eqref{trace2} and \eqref{der:M} to obtain
	\begin{align*}
		& (m-1)!\Tr\big(\wh D_m(\lm)\big)
		=
		\Tr\Biggl(\,\sum_{\substack{p+q=m-1 \\ p,q\ge0}}\binom{m-1}{p}
		\wh\gamma^{(p)}(\lambda) S^{(q)}(\lambda)\Biggr)
		\\
		&= \sum_{\substack{p+q=m-1 \\ p,q\ge0}}\binom{m-1}{p}
		\Tr\Bigl(\wh\gamma^{(p)}(\lambda) S^{(q)}(\lambda)\Bigr)
		\\
		&= \sum_{\substack{p+q=m-1 \\ p,q\ge0}}\binom{m-1}{p}
		\Tr\Bigl(S^{(q)}(\lambda)\wh\gamma^{(p)}(\lambda)\Bigr)
		\displaybreak[0]\\
		&=
		\Tr\Biggl(\,\sum_{\substack{p+q=m-1 \\ p,q\ge0}}\binom{m-1}{p}
		S^{(q)}(\lambda)\wh\gamma^{(p)}(\lambda)\Biggr)
		\\
		&= \Tr\biggl(\frac{\dd^{m-1}}{\dd\lambda^{m-1}}\Bigl(S(\lambda)
		\wh\gamma(\lambda)\Bigr)\biggr)
		\displaybreak[0]\\
		&= \Tr\biggl(\frac{\dd^{m-1}}{\dd\lambda^{m-1}}
		\Bigl(\wh M(\lambda)^{-1}\wh\gamma(\ov\lambda)^*\wh \gamma(\lambda)\Bigr)\biggr)
		\\
		&=
		\Tr\biggl(\frac{\dd^{m-1}}{\dd\lambda^{m-1}}\Bigl(\wh M(\lambda)^{-1}\wh M'(\lambda)\Bigr)\biggr),
	\end{align*}
	which finishes the proof.
\end{proof}

\begin{proof}[Proof of Theorem~\ref{thm2}]
	(i) The proof of this statement is fully
	analogous to the proof of Theorem~\ref{thm1}.
	One has to replace in the argument
	$\Opfree$, $\alpha$, $\Op$, $\wt M$, $\wt\gamma$,
	by $\OpN$, $\omega$, $\Opp$, $\wh M$, $\wh\gamma$, respectively,
	Moreover, Krein's resolvent formula is used in \eqref{eq:Krein3}
	instead of  Krein's formula in \eqref{eq:Krein2}
	and Lemma~\ref{lem:der2} instead of Lemma~\ref{lem:der1}.

	(ii) By item~(i) of this theorem and by Lemma~\ref{lem:aux},
	for every $m\in\dN$ such that $m > \frac{d-1}{2}$ and every
	$\lm \in \rho(\Opp)$ both operators
	$\wh D_m(\lm)$ and $\wh D_{\omega, m}(\lm)$ belong to the trace class.
	In view of the identity
	$\wh E_{\omega, m}(\lm) = \wh D_m(\lm) + \wh D_{\omega, m}(\lm)$,
	we infer that $\wh E_{\omega, m}(\lm)$ is also in the trace class.
	Using the formula \eqref{trace1} we have
	\[
		\Tr(\wh E_{\omega,m}(\lm))
		=
		\Tr(\wh D_{\omega,m}(\lm)) + \Tr(\wh D_m(\lm)).
	\]
	Combining the trace formula in (i) of this theorem and
	the trace formula in Lemma~\ref{lem:aux} we obtain
	\begin{align*}
		\Tr(\wh E_{\omega,m}(\lm))
		& =
		\frac{1}{(m-1)!}
		\Tr\bigg(\frac{\dd^{m-1}}{\dd \lm^{m-1}}
		\Big(\big(I - \omega\wh M(\lm)\big)^{-1}\omega
		\wh M'(\lm) + \wh M(\lm)^{-1}\wh M'(\lm) \Big) \bigg)\\
		& =
		\frac{1}{(m-1)!}
		\Tr\bigg(\frac{\dd^{m-1}}{\dd \lm^{m-1}}
		\Big(\big(I - \omega\wh M(\lm)\big)^{-1}\wh M(\lm)^{-1}\wh M'(\lm)	
		\bigg),
	\end{align*}
	which finishes the proof.
\end{proof}

\end{document}